\documentclass[reqno]{amsart}

\usepackage{epsfig}
\usepackage{graphics}
\usepackage{amsmath}
\usepackage{amsthm}
\usepackage{amssymb,enumerate}
\usepackage{latexsym}

\newtheorem{theorem}{Theorem}[section]
\newtheorem{lemma}[theorem]{Lemma}
\theoremstyle{definition}
\newtheorem{definition}[theorem]{Definition}
\newtheorem{example}{Example}[section]

\begin{document}

\title{A formulation of Noether's theorem for fuzzy problems of the
calculus of variations}

\author{J. Soolaki, O.S. Fard, R. Almeida \and A.H. Borzabadi}

\address{Javad Soolaki \newline
\indent Department of Mathematics, Damghan University, Damghan, Iran}
\email{javad.soolaki@gmail.com}

\address{Omid Solaymani Fard\newline
\indent School of Mathematics and Computer Science, Damghan University, Damghan, Iran}
\email{osfard@du.ac.ir, omidsfard@gmail.com}

\address{Ricardo Almeida\newline
\indent Center for Research and Development in Mathematics and Applications (CIDMA)\\
Department of Mathematics, University of Aveiro, 3810--193 Aveiro, Portugal}
\email{ricardo.almeida@ua.pt}

\address{Akbar Hashemi Borzabadi\newline
\indent School of Mathematics and Computer Science, Damghan University, Damghan, Iran}
\email{borzabadi@du.ac.ir}
% --------------------------------------------



% -------------------------------------------
\begin{abstract}
The theory of the calculus of variations for fuzzy
systems was recently initiated in \cite{Farhadinia}, with the proof of the fuzzy Euler--Lagrange equation. Using fuzzy Euler--Lagrange equation,   we obtain here a Noether--like theorem
for fuzzy variational problems.
\end{abstract}

\maketitle

% ----------------------------------------------------------------------

\section{Introduction}
The notion of conservation law is well
known in Physics. Many  examples of conservation laws appear in modern physics: in classic,
quantum, and optical mechanics; in the theory of relativity, etc \cite{Kosmann,Torres1}. The   conservation laws  can materially simplify
the problem of finding extremals when the order of the conservation law is less
than that for the corresponding Euler--Lagrange equation. But it is not obvious how one might derive
a conservation law. Some functionals may have several conservation laws; others may have no
conservation laws. A central result called Noether's theorem links conservation laws with
certain invariance properties of the functional, and it provides an algorithm
for finding the conservation law.
 In the last decades, Noether's principle has been formulated in various contexts
(see \cite{Torres,Torres1} and references therein).  In this work we generalize Noether's theorem for fuzzy variational problems.\par
The fuzzy calculus of variations extends
the classical variational calculus by considering fuzzy variables and their derivatives into the variational integrals to be extremized. This may  occur naturally in many problems of physics and mechanics. Very few works has been done to the fuzzy variational problems. Recently, Farhadinia \cite{Farhadinia} studied necessary optimality conditions for fuzzy variational problems by using the fuzzy differentiability concept due to Buckley and Feuring \cite{Buckley}. In \cite{Fard},  Fard and
Zadeh by using $\alpha$--differentiability concept obtained an extended fuzzy
Euler-Lagrange condition.  Fard et al.\cite{Fard1} presented the fuzzy Euler--Lagrange conditions for fuzzy constrained and unconstrained variational problems under the generalized Hukuhara differentiability. Here we use the results of \cite{Farhadinia,Fard,Fard1}, to generalize
Noether's theorem for the more general context of the fuzzy calculus of variations.\par
The paper is organized as follows. Section \ref{1} presents some preliminaries needed in the sequel. In Section \ref{2} fuzzy Noether's theorem is formulated and proved. The method is based on a two-step procedure: it starts with an invariance
notion of the integral functional under a one--parameter infinitesimal group of transformations,
without changing the time variable; then it proceeds with a time--reparameterization to obtain
Noether's theorem in general form. We discuss the applicability
of the main theorems through an example in Section \ref{3}. Finally, we give a  conclusion in Section \ref{4}.

% --------------------------------------------

\section{Preliminaries}
\label{1}
The fuzzy number $\tilde{a}:\mathbb{R}\longrightarrow [0,1]$ is a mapping with the properties: (i) $\tilde{a}$ is normal, i.e., there exists an $x \in \mathbb{R}$ such that  $\tilde{a}(x)=1$; (ii)  $\tilde{a}$ is fuzzy convex, i.e., $\tilde{a}(\lambda x +(1-\lambda) y) \geq \min \{\tilde{a}(x), \tilde{a}(y) \}$ for all $\lambda \in [0,1], x,y \in[0,1];$  (iii) $\tilde{a}$ is upper semicontinuous, i.e.,  $\tilde{a} (x_0) \geq \overline{lim}_{k\rightarrow \infty}\tilde{a} (x_{k})$  for any $x_k\in \mathbb{R},$ as $x_k \rightarrow x_0;$ (iv) the support of $\tilde{a}$ which is $supp (\tilde{a}) = cl\{ x\in \mathbb{R}: \tilde{a} (x)>0\} $ is compact.\par
We denote by  $\mathcal{F}_R$ the set of all fuzzy numbers on $\mathbb{R}$. The $r-$level set of  $\tilde{a}\in \mathcal{F}_R$,  denoted by $[\tilde{a}]^r$ is defined by $[\tilde{a}]^r=\{ x\in \mathbb{R}: \tilde{a} (x)\geq r \}$ for all  $r\in[0,1].$ The 0-level set $[\tilde{a}]^0$ is defined as the closure of  $\{ x\in \mathbb{R}: \tilde{a} (x)\geq 0 \}$, i.e., $[\tilde{a}]^0= cl(supp(\tilde{a}))$.\par
Obviously, the  $r$-level set $[\tilde{a}]^r=[\underline{a}^r,\overline{a}^r]$ is a closed interval in $\mathbb{R}$ for all $r\in[0,1]$, where $\underline{a}^r$ and $\overline{a}^r$ denote the left-hand
and right-hand endpoints of $[\tilde{a}]^r$, respectively. Needless to say that $\tilde{a}$ is a crisp number with value $k$ if its membership
function is given by $\tilde{a}(x)=1$ if $x=k,$ and $\tilde{a}(x)=0$ otherwise. Also we define fuzzy zero as$$
\tilde{0}_{x}=
\begin{cases}
1 & \text{ if } x=0,\\
0 & \text{ if } x\neq 0.
\end{cases}
$$
By the following lemma, we present some interesting properties associated to   $\underline{a}^r$ and $\overline{a}^r$ of a fuzzy number
$\tilde{a}\in \mathcal{F}_R$.
\begin{lemma}[See Theorem~1.1 of \cite{MR0825618} and Lemma~2.1 of \cite{MR2609258}]\label{lem1}
If $\underline{a}^r:[0,1] \rightarrow \mathbb{R}$ and
$\overline{a}^r:[0,1] \rightarrow \mathbb{R}$
satisfy the conditions
\begin{enumerate}
\item[(i)]  $\underline{a}^r:[0,1] \rightarrow \mathbb{R} $ is a bounded nondecreasing function,
\item[(ii)]$\overline{a}^r:[0,1] \rightarrow \mathbb{R} $ is a bounded nonincreasing function,
\item[(iii)]  $\underline{a}^1\leq \overline{a}^1$,
\item[(iv)] for $0<k\leq 1$, $\lim_{r\rightarrow k^{-}}\underline{a}^r =\underline{a}^k$
and $\lim_{r\rightarrow k^{-}}\overline{a}^r =\overline{a}^k$,
\item[(v)] $\lim_{r\rightarrow 0^{+}}\underline{a}^r =\underline{a}^0$
and $\lim_{r\rightarrow 0^{+}}\overline{a}^r =\overline{a}^0$,
\end{enumerate}
then $\tilde{a}:\mathbb{R}\rightarrow [0,1]$, characterized by
$\tilde{a}(t)=\sup\{r|\underline{a}^r\leq t \leq \overline{a}^r\}$,
is a fuzzy number with $[\tilde{a}]^{r}=[\underline{a} ^r,\overline{a}^r]$.
The converse is also true: if $\tilde{a}(t)=\sup\{r|\underline{a}^r\leq t\leq  \overline{a}^r\}$ is
a fuzzy number with parametrization given by $ [\tilde{a}]^{r}=[\underline{a} ^r,\overline{a}^r]$,
then functions $\underline{a}^r$ and $\overline{a}^r$ satisfy conditions (i)--(v).
\end{lemma}
For $\tilde{a}, \tilde{b}\in\mathcal{F}_R$ and $\lambda \in\mathbb{R}$,
the sum $\tilde{a}+\tilde{b}$ and the product $\lambda \cdot \tilde{a}$
are defined by $[\tilde{a}+\tilde{b}]^{r}=[\tilde{a}]^{r}+[\tilde{b}]^{r}$
and $[\lambda \cdotp \tilde{a}]^r=\lambda[\tilde{a}]^r$ for all $r\in[0,1]$,
where $[\tilde{a}]^{r}+[\tilde{b}]^{r}$ means the usual addition of two intervals
(subsets) of $\mathbb{R}$ and $\lambda[\tilde{a}]^r$ means the usual product
between a scalar and a subset of $\mathbb{R}$. The product $\tilde{a} \odot \tilde{b}$
of fuzzy numbers $\tilde{a}$ and $\tilde{b}$, is defined by
$[\tilde{a}\odot\tilde{b}]^r=[\min\{ \underline{a}^r \underline{b}^r,
\underline{a}^r \overline{b}^r, \overline{a}^r \underline{b}^r,
\overline{a}^r \overline{b}^r \},
\max \{\underline{a}^r \underline{b}^r, \underline{a}^r \overline{b}^r,
\overline{a}^r \underline{b}^r, \overline{a}^r \overline{b}^r \}]$.
The metric structure is given by the Hausdorff distance
$D:\mathcal{F}_R \times \mathcal{F}_R \rightarrow \mathbb{R}_{+}\cup\{0\}$,
$D(\tilde{a},\tilde{b})=\sup_{r\in[0,1]}
\max\{|\underline{a}^r-\underline{b}^r|,|\overline{a}^r-\overline{b}^r|\}$.\par
A triangular fuzzy number, denoted by $\tilde{a}=(x,y,z)$ where $x\leq y \leq z,$  has $r$-level set $[yr+x(1-r),yr+z(1-r)],~r\in[0,1].$
\begin{definition}[See \cite{Farhadinia}]
We say that $\tilde{f}:[a,b]\rightarrow \mathcal{F}_R$ is continuous
at $x\in[a,b]$, if both $\underline{f}^r(x)$ and $\overline{f}^r(x)$
are continuous functions of $x\in[a,b]$ for all $r \in [0,1]$.
\end{definition}
\begin{definition}[See \cite{Hoa}]
The generalized Hukuhara difference of two fuzzy numbers
$\tilde{a},\tilde{b}\in\mathcal{F}_R$ ($gH$-difference for short) is defined as follows:
\begin{equation*}
\tilde{a}\ominus_{gH}\tilde{b}=\tilde{c}
\Leftrightarrow
\tilde{a}=\tilde{b}+\tilde{c}
\text{ or }
\tilde{b}=\tilde{a}+(-1)\tilde{c}.
\end{equation*}
\end{definition}
If $\tilde{c}= \tilde{a}\ominus_{gH}\tilde{b}$ exists as a fuzzy number, then
its level cuts $[\underline{c}^r, \overline{c}^r]$ are obtained by
$\underline{c}^r=\min\{ \underline{a}^r - \underline{b}^r,
\overline{a}^r - \overline{b}^r\}$ and $\overline{c}^r
=\max\{ \underline{a}^r - \underline{b}^r,
\overline{a}^r - \overline{b}^r\}$ for all $r\in[0,1]$.

\begin{definition}[See \cite{Hoa}]
\label{def02}
Let $x\in(a,b)$ and $h$ be such that $x+h\in(a,b)$.
The generalized Hukuhara derivative of a fuzzy--valued function
$\tilde{f}:(a,b) \rightarrow \mathcal{F}_R$ at $x$ is defined by
\begin{equation}
\label{G2}
\mathcal{D}_{gH}\tilde{f}(x)=\lim_{{h\to 0}}\frac{\tilde{f}(x+h)\ominus_{gH} \tilde{f}(x)}{h}.
\end{equation}
If $\mathcal{D}_{gH}\tilde{f}(x) \in \mathcal{F}_R$ satisfying \eqref{G2} exists,
then we say that $\tilde{f}$ is generalized Hukuhara differentiable
($gH$-differentiable for short) at $x$. Also, we say that $\tilde{f}$ is
$[(1)-gH]$-differentiable at $x$ (denoted by $\mathcal{D}_{1,gH} \tilde{f}$)
if $[\mathcal{D}_{gH}\tilde{f}(x)]^r = [\dot{\underline{f}}^r(x), \dot{\overline{f}}^r(x)]$,
and that $\tilde{f}$ is $[(2)-gH]$-differentiable at $x$ (denoted by $\mathcal{D}_{2,gH} \tilde{f}$)
if $[\mathcal{D}_{gH}\tilde{f}(x)]^r = [\dot{\overline{f}}^r(x), \dot{\underline{f}}^r(x)]$, $r\in[0,1]$.
\end{definition}
If the fuzzy function $\tilde{f}(x)$ is continuous in the metric $D$,
then its definite integral exists. Furthermore,
\begin{equation*}
\left( \underline{\int_{a}^{b} \tilde{f}(x)dx} \right)^r=\int_{a}^{b} \underline{f}^r (x)dx,
\quad \left( \overline{\int_{a}^{b} \tilde{f}(x)dx} \right)^r=\int_{a}^{b} \overline{f}^r (x)dx.
\end{equation*}
\begin{definition}[See \cite{Farhadinia}]
\label{def1}
Let $\tilde{a},\tilde{b}\in\mathcal{F}_R$. We write $\tilde{a}\preceq\tilde{b}$,
if $\underline{a}^r \leq \underline{b}^r $ and $\overline{a}^r \leq \overline{b}^r$
for all $r \in [0,1]$. We also write $\tilde{a}\prec\tilde{b}$,
if $\tilde{a}\preceq\tilde{b}$ and there exists an $r'\in [0,1]$
so that $\underline{a}^{r'}  < \underline{b}^{r'}$ and
$\overline{a}^{r'}  < \overline{b}^{r'}$.  Moreover,
$\tilde{a}\approx\tilde{b}$ if $\tilde{a}\preceq\tilde{b}$
and $\tilde{a}\succeq\tilde{b}$, that is,
$[\tilde{a}]^r=[\tilde{b}]^r$ for all $r \in[0,1]$.
\end{definition}
We say that $\tilde{a},\tilde{b}\in\mathcal{F}_R$
are comparable if either $\tilde{a}\preceq\tilde{b}$
or $\tilde{a}\succeq\tilde{b}$; and noncomparable otherwise.

% --------------------------------------------
\section{Fuzzy Noether's theorem}\label{2}
There exist several ways to prove the classical theorem of Emmy Noether. In this section we
extend one of those proofs \cite{Torres2}. The proof is done in two steps: we begin by proving Noether's theorem  without transformation of the independent variable;
then  we obtain Noether's theorem in its general
form.\par
In 2010 \cite{Farhadinia}, a formulation of the Euler--Lagrange equations was given for problems of the
fuzzy calculus of variations. In this section we prove a Noether's theorem
for the fuzzy Euler--Lagrange extremals. Along the work, we denote by $\partial_i \underline{L}^r~(\partial_i \overline{L}^r)$ the partial derivative of $ \underline{L}^r ~(\overline{L}^r)$ with respect to its $i$th argument.\par
 The fundamental functional of the fuzzy calculus of variations is defined as follows:
\begin{equation}\label{z1}
\tilde{I}[\tilde{q}(.)]=\int_{a}^{b} \tilde{L} (x,\tilde{q}(x),\dot{\tilde{q}}(x))dx\longrightarrow \min,
\end{equation}
under the boundary conditions $\tilde{q}(a)=\tilde{q}_a$ and $\tilde{q}(b)=\tilde{q}_b$.
The $r$--level set of Lagrangian $\tilde{L} : [a,b]\times\mathcal{F}_R \times\mathcal{F}_R  \rightarrow \mathcal{F}_R$  is
\begin{multline*}
\left[\tilde{L}(x,\tilde{q}(x),\dot{\tilde{q}}(x))\right]^r\\
=\left[\underline{L}^r\left(x,\underline{q}^r(x),\overline{q}^r(x),\dot{\underline{q}}^r(x),\dot{\overline{q}}^r(x)\right),\overline{L}^r\left(x,\underline{q}^r(x),\overline{q}^r(x),\dot{\underline{q}}^r(x),\dot{\overline{q}}^r(x)\right)\right].
\end{multline*}
The Lagrangian $\underline{L}^r$ and $\overline{L}^r$ are  assumed to be $C^2$--functions with respect to all its arguments.
\begin{theorem}[See\cite{Farhadinia}]\label{t1}
If $\tilde{q}(x)$ is a minimizer of problem (\ref{z1}), then it satisfies the fuzzy Euler--Lagrange equations:
\begin{equation}\label{q7}
\begin{gathered}
\partial_2\underline{L}^r\left(x,\underline{q}^r(x),\overline{q}^r(x),\dot{\underline{q}}^r(x),\dot{\overline{q}}^r(x)\right)-\frac{d}{dx}\partial_4\underline{L}^r\left(x,\underline{q}^r(x),\overline{q}^r(x),\dot{\underline{q}}^r(x),\dot{\overline{q}}^r(x)\right)=0,\\
\partial_3\underline{L}^r\left(x,\underline{q}^r(x),\overline{q}^r(x),\dot{\underline{q}}^r(x),\dot{\overline{q}}^r(x)\right)-\frac{d}{dx}\partial_5\underline{L}^r\left(x,\underline{q}^r(x),\overline{q}^r(x),\dot{\underline{q}}^r(x),\dot{\overline{q}}^r(x)\right)=0,\\
\partial_2\overline{L}^r\left(x,\underline{q}^r(x),\overline{q}^r(x),\dot{\underline{q}}^r(x),\dot{\overline{q}}^r(x)\right)-\frac{d}{dx}\partial_4\overline{L}^r\left(x,\underline{q}^r(x),\overline{q}^r(x),\dot{\underline{q}}^r(x),\dot{\overline{q}}^r(x)\right)=0,\\
\partial_3\overline{L}^r\left(x,\underline{q}^r(x),\overline{q}^r(x),\dot{\underline{q}}^r(x),\dot{\overline{q}}^r(x)\right)-\frac{d}{dx}\partial_5\overline{L}^r\left(x,\underline{q}^r(x),\overline{q}^r(x),\dot{\underline{q}}^r(x),\dot{\overline{q}}^r(x)\right)=0,
\end{gathered}
\end{equation}
for all $r\in[0,1].$
\end{theorem}
\begin{definition}[Invariance  without transforming time]\label{q26}
Functional $(\ref{z1})$ is said to be invariant
under an $\epsilon$--parameter group of infinitesimal transformations
\begin{equation}\label{q1}
\tilde{\hat{q}}(x)=\tilde{q}(x)+\epsilon\tilde{\zeta}(x,\tilde{q}(x))+ o(\epsilon)
\end{equation}
if  and only if,
\begin{equation}\label{q2}
\int_{t_a}^{t_b} \tilde{L} \left(x,\tilde{q}(x),\dot{\tilde{q}}(x)\right)dx= \int_{t_a}^{t_b} \tilde{L} \left(x,\tilde{\hat{q}}(x),\dot{\tilde{\hat{q}}}(x)\right)dx,
\end{equation}
for any subinterval $[t_a,t_b]\subseteq [a,b]$ and $\epsilon>0.$
\end{definition}
 Follows from the definition of partial ordering given in Definition \ref{def1}, the
inequality (\ref{q2}) holds if and only if
\begin{equation}\label{q3}
\int_{t_a}^{t_b} \underline{L}^r \left(x,\underline{q}^r(x),\overline{q}^r(x),\dot{\underline{q}}^r(x),\dot{\overline{q}}^r(x)\right)dx= \int_{t_a}^{t_b} \underline{L}^r \left(x,\underline{\hat{q}}^r(x),\overline{\hat{q}}^r(x),\dot{\underline{\hat{q}}}^r(x),\dot{\overline{\hat{q}}}^r(x)\right)dx,
\end{equation}
and
\begin{equation}\label{q4}
\int_{t_a}^{t_b} \overline{L}^r \left(x,\underline{q}^r(x),\overline{q}^r(x),\dot{\underline{q}}^r(x),\dot{\overline{q}}^r(x)\right)dx= \int_{t_a}^{t_b} \overline{L}^r \left(x,\underline{\hat{q}}^r(x),\overline{\hat{q}}^r(x),\dot{\underline{\hat{q}}}^r(x),\dot{\overline{\hat{q}}}^r(x)\right)dx,
\end{equation}
for all $r\in[0,1],$
where
\begin{equation}\label{q5}
\begin{gathered}
\underline{\hat{q}}^r(x)=\underline{q}^r(x)+\epsilon\underline{\zeta}^r(x,\underline{q}^r,\overline{q}^r)+ o(\epsilon),\\
\overline{\hat{q}}^r(x)=\overline{q}^r(x)+\epsilon\overline{\zeta}^r(x,\underline{q}^r,\overline{q}^r)+ o(\epsilon).
\end{gathered}
\end{equation}
\begin{theorem}[Necessary condition of invariance]\label{t2} If functional \eqref{z1} is invariant under transformations \eqref{q1}, then
\begin{equation}\label{q13}
\partial_2 \underline{L}^r.\underline{\zeta}^r+\partial_3 \underline{L}^r.\overline{\zeta}^r+\partial_4 \underline{L}^r.\dot{\underline{\zeta}}^r+\partial_5\underline{L}^r.\dot{\overline{\zeta}}^r=0,
\end{equation}
\begin{equation}\label{q14}
\partial_2 \overline{L}^r.\underline{\zeta}^r+\partial_3 \overline{L}^r.\overline{\zeta}^r+\partial_4 \overline{L}^r.\dot{\underline{\zeta}}^r+\partial_5\overline{L}^r.\dot{\overline{\zeta}}^r=0.
\end{equation}
for all $r\in[0,1]$.
\end{theorem}
\begin{proof}
Eqs. \eqref{q3} and \eqref{q4} are equivalent to
\begin{multline}\label{q11}
\underline{L}^r\left(x,\underline{q}^r(x),\overline{q}^r(x),\dot{\underline{q}}^r(x),\dot{\overline{q}}^r(x)\right)=  \underline{L}^r\left(x,\underline{q}^r(x)+\epsilon\underline{\zeta}^r(x)+o(\epsilon),\right.\\\left.
\overline{q}^r(x)+\epsilon\overline{\zeta}^r+o(\epsilon),\dot{\underline{q}}^r(x)+\epsilon\dot{\underline{\zeta}}^r+o(\epsilon),\dot{\overline{q}}^r(x)+
\epsilon\dot{\overline{\zeta}}^r+o(\epsilon)\right)
\end{multline}
and
\begin{multline}\label{q12}
\overline{L}^r\left(x,\underline{q}^r(x),\overline{q}^r(x),\dot{\underline{q}}^r(x),\dot{\overline{q}}^r(x)\right)=  \overline{L}^r\left(x,\underline{q}^r+\epsilon\underline{\zeta}^r+o(\epsilon)\right.,\\\left.\overline{q}^r(x)+\epsilon\overline{\zeta}^r+o(\epsilon),\dot{\underline{q}}^r(x)+\epsilon\dot{\underline{\zeta}}^r+o(\epsilon),\dot{\overline{q}}^r(x)+
\epsilon\dot{\overline{\zeta}}^r+o(\epsilon)\right),
\end{multline}
for all  $r\in[0,1]$. Differentiating both sides of Eqs.(\ref{q11}) and (\ref{q12}) with respect to $\epsilon$ then substituting $\epsilon=0$, we obtain
(\ref{q13}) and (\ref{q14}).
\end{proof}
\begin{definition}[Conserved quantity]\label{def0} Quantity $\tilde{C}(x,\tilde{q}(x),\dot{\tilde{q}}(x))$ is said to be conserved if, and only if,
\begin{equation}
\frac{d}{dx}\underline{C}^r\left(x,\underline{q}^r(x),\overline{q}^r(x),\dot{\underline{q}}^r(x),\dot{\overline{q}}^r(x)\right)=\frac{d}{dx}\overline{C}^r\left(x,\underline{q}^r(x),\overline{q}^r(x),\dot{\underline{q}}^r(x),\dot{\overline{q}}^r(x)\right)=0,
\end{equation}
 along all the solutions of the Euler--Lagrange equations (\ref{q7}) and  for all $r\in[0,1]$.
\end{definition}
\begin{theorem}
[Noether's theorem without transforming time] \label{q27}
If functional \eqref{z1} is invariant under
the one--parameter group of transformations \eqref{q1}, then $\tilde{C}(x,\tilde{q}(x),\dot{\tilde{q}}(x))$  is conserved where the lower and upper bound of  $\tilde{C}$ are
\begin{multline}
\underline{C}^r\left(x,\underline{q}^r(x),\overline{q}^r(x),\dot{\underline{q}}^r(x),\dot{\overline{q}}^r(x)\right)=\partial_4 \underline{L}^r\left(x,\underline{q}^r(x),\overline{q}^r(x),\dot{\underline{q}}^r(x),\dot{\overline{q}}^r(x)\right)\\
.\underline{\zeta}^r(x,\underline{q}^r(x),\overline{q}^r(x))+\partial_5 \underline{L}^r\left(x,\underline{q}^r(x),\overline{q}^r(x),\dot{\underline{q}}^r(x),\dot{\overline{q}}^r(x)\right).\overline{\zeta}^r(x,\underline{q}^r(x),\overline{q}^r(x))
\end{multline}
and
\begin{multline}
\overline{C}^r\left(x,\underline{q}^r(x),\overline{q}^r(x),\dot{\underline{q}}^r(x),\dot{\overline{q}}^r(x)\right)=\partial_4 \overline{L}^r\left(x,\underline{q}^r(x),\overline{q}^r(x),\dot{\underline{q}}^r(x),\dot{\overline{q}}^r(x)\right)\\.\underline{\zeta}^r(x,\underline{q}^r(x),\overline{q}^r(x))+\partial_5 \overline{L}^r\left(x,\underline{q}^r(x),\overline{q}^r(x),\dot{\underline{q}}^r(x),\dot{\overline{q}}^r(x)\right).\overline{\zeta}^r(x,\underline{q}^r(x),\overline{q}^r(x))
\end{multline}
 for all $r\in[0,1]$.
\end{theorem}
\begin{proof}
Using the  Euler--Lagrange equations (\ref{q7}) and the necessary conditions of invariance (\ref{q13}) and (\ref{q14}),
we obtain:
\begin{equation}\label{q17}
\begin{gathered}
\frac{d}{dx}
\left[\partial_4 \underline{L}^r\left(x,\underline{q}^r,\overline{q}^r,\dot{\underline{q}}^r,\dot{\overline{q}}^r\right).\underline{\zeta}^r(x,\underline{q}^r,\overline{q}^r)
+\partial_5\underline{L}^r\left(x,\underline{q}^r,\overline{q}^r,\dot{\underline{q}}^r,\dot{\overline{q}}^r\right).
\overline{\zeta}^r(x,\underline{q}^r,\overline{q}^r)\right]\\
=\frac{d}{dx}\partial_4 \underline{L}^r\left(x,\underline{q}^r,\overline{q}^r,\dot{\underline{q}}^r,\dot{\overline{q}}^r\right).\underline{\zeta}^r(x,\underline{q}^r,\overline{q}^r)
+\partial_4 \underline{L}^r\left(x,\underline{q}^r,\overline{q}^r,\dot{\underline{q}}^r,
\dot{\overline{q}}^r\right).\dot{\underline{\zeta}}^r(x,\underline{q}^r,\overline{q}^r)\\
+\frac{d}{dx}\partial_5 \underline{L}^r\left(x,\underline{q}^r,\overline{q}^r,\dot{\underline{q}}^r,
\dot{\overline{q}}^r\right).\overline{\zeta}^r(x,\underline{q}^r,\overline{q}^r)
+\partial_5 \underline{L}^r\left(x,\underline{q}^r,\overline{q}^r,
\dot{\underline{q}}^r,\dot{\overline{q}}^r\right).\dot{\overline{\zeta}}^r(x,\underline{q}^r,\overline{q}^r)\\
=\partial_2 \underline{L}^r\left(x,\underline{q}^r,\overline{q}^r,
\dot{\underline{q}}^r,\dot{\overline{q}}^r\right).\underline{\zeta}^r(x,\underline{q}^r,\overline{q}^r)
+\partial_3 \underline{L}^r\left(x,\underline{q}^r,\overline{q}^r,
\dot{\underline{q}}^r,\dot{\overline{q}}^r\right).\overline{\zeta}^r(x,\underline{q}^r,\overline{q}^r)\\
+\partial_4 \underline{L}\left(x,\underline{q}^r,\overline{q}^r,
\dot{\underline{q}}^r,\dot{\overline{q}}^r\right).\dot{\underline{\zeta}}^r(x,\underline{q}^r,\overline{q}^r)
+\partial_5 \underline{L}\left(x,\underline{q}^r,\overline{q}^r,
\dot{\underline{q}}^r,\dot{\overline{q}}^r\right).\dot{\overline{\zeta}}^r(x,\underline{q}^r,\overline{q}^r)=0.
\end{gathered}
\end{equation}
Computing similar to those in (\ref{q17}), one can easily verify
\begin{equation}\label{q16}
\frac{d}{dx}\left[\partial_4 \overline{L}^r\left(x,\underline{q}^r,\overline{q}^r,\dot{\underline{q}}^r,\dot{\overline{q}}^r\right).\underline{\zeta}^r(x,\underline{q}^r,\overline{q}^r)+\partial_5 \overline{L}^r\left(x,\underline{q}^r,\overline{q}^r,\dot{\underline{q}}^r,\dot{\overline{q}}^r\right).\overline{\zeta}^r(x,\underline{q}^r,\overline{q}^r)\right]=0.
\end{equation}
\end{proof}
\begin{definition}[Invariance of (\ref{z1})]\label{q18}
 Functional $(\ref{z1})$ is said to be invariant under the one-parameter
group of infinitesimal transformations
\begin{equation}
\begin{gathered}
\hat{x}=x+\epsilon\tau(x,\tilde{q})+o(\epsilon),\\
\tilde{\hat{q}}(x)=\tilde{q}(x)+\epsilon\tilde{\zeta}(x,\tilde{q})+o(\epsilon)
\end{gathered}
\end{equation}
if, and only if,
\begin{equation}\label{q20}
\begin{gathered}
\int_{t_a}^{t_b} \underline{L}^r \left(x,\underline{q}^r(x),\overline{q}^r(x),\dot{\underline{q}}^r(x),\dot{\overline{q}}^r(x)\right)dx= \int_{\hat{x}(t_a)}^{\hat{x}(t_b)} \underline{L}^r \left(\hat{x},\underline{\hat{q}}^r(\hat{x}),\overline{\hat{q}}^r(\hat{x}),\dot{\underline{\hat{q}}}^r(\hat{x}),\dot{\overline{\hat{q}}}^r(\hat{x})\right)d\hat{x},\\
\int_{t_a}^{t_b} \overline{L}^r \left(x,\underline{q}^r(x),\overline{q}^r(x),\dot{\underline{q}}^r(x),\dot{\overline{q}}^r(x)\right)dx= \int_{\hat{x}(t_a)}^{\hat{x}(t_b)} \overline{L}^r \left(\hat{x},\underline{\hat{q}}^r(\hat{x}),\overline{\hat{q}}^r(\hat{x}),\dot{\underline{\hat{q}}}^r(\hat{x}),\dot{\overline{\hat{q}}}^r(\hat{x})\right)d\hat{x},
\end{gathered}
\end{equation}
for any subinterval $[t_a,t_b]\subseteq [a,b]$ and for all $r\in[0,1]$  and $\epsilon>0.$
\end{definition}
\begin{theorem}[Noether's theorem] \label{q36}
If functional \eqref{z1} is invariant, in the sense of Definition \eqref{q18},
then $\tilde{C}(x,\tilde{q}(x),\dot{\tilde{q}}(x))$ is conserved for all $r\in[0,1]$, where the lower and upper bound of $\tilde{C}$ are
\begin{equation}\label{q32}
\begin{gathered}
\underline{C}^r\left(x,\underline{q}^r,\overline{q}^r,\dot{\underline{q}}^r,\dot{\overline{q}}^r\right)=\partial_4 \underline{L}^r\left(x,\underline{q}^r,\overline{q}^r,\dot{\underline{q}}^r,\dot{\overline{q}}^r\right).\underline{\zeta}^r(x,\underline{q}^r,\overline{q}^r)+\partial_5 \underline{L}^r\left(x,\underline{q}^r,\overline{q}^r,\dot{\underline{q}}^r,\dot{\overline{q}}^r\right)\\.\overline{\zeta}^r(x,\underline{q}^r,\overline{q}^r)
+\left[ \underline{L}^r\left(x,\underline{q}^r,\overline{q}^r,\dot{\underline{q}}^r,\dot{\overline{q}}^r\right)-\partial_4\underline{L}^r\left(x,\underline{q}^r,\overline{q}^r,\dot{\underline{q}}^r,\dot{\overline{q}}^r\right).\dot{\underline{q}}^r\right.\\\left.
-\partial_5\underline{L}^r
\left(x,\underline{q}^r,\overline{q}^r,\dot{\underline{q}}^r,\dot{\overline{q}}^r\right).\dot{\overline{q}}^r\right].\tau(x,\underline{q}^r,\overline{q}^r)
\end{gathered}
\end{equation}
and
\begin{equation}\label{q33}
\begin{gathered}
\overline{C}^r\left(x,\underline{q}^r,\overline{q}^r,\dot{\underline{q}}^r,\dot{\overline{q}}^r\right)=\partial_4 \overline{L}^r\left(x,\underline{q}^r,\overline{q}^r,\dot{\underline{q}}^r,\dot{\overline{q}}^r\right).\underline{\zeta}^r(x,\underline{q}^r,\overline{q}^r)+\partial_5 \overline{L}^r\left(x,\underline{q}^r,\overline{q}^r,\dot{\underline{q}}^r,\dot{\overline{q}}^r\right)\\.\overline{\zeta}^r(x,\underline{q}^r,\overline{q}^r)
+\left[ \overline{L}^r\left(x,\underline{q}^r,\overline{q}^r,\dot{\underline{q}}^r,\dot{\overline{q}}^r\right)-\partial_4\overline{L}^r\left(x,\underline{q}^r,\overline{q}^r,\dot{\underline{q}}^r,\dot{\overline{q}}^r\right).\dot{\underline{q}}^r\right.\\\left.
-\partial_5\overline{L}^r
\left(x,\underline{q}^r,\overline{q}^r,\dot{\underline{q}}^r,\dot{\overline{q}}^r\right).\dot{\overline{q}}^r\right].\tau(t,\underline{q}^r,\overline{q}^r)
\end{gathered}
\end{equation}
\end{theorem}
\begin{proof}
Every non-autonomous problem (\ref{z1}) is equivalent to an autonomous one, considering $x$ as
a dependent variable. For that we consider a Lipschitzian one-to-one transformation
\begin{equation}
[a,b]\ni x \rightarrow \sigma \in [\sigma_a,\sigma_b]
\end{equation}
such that
\begin{align*}
\underline{I}^r\left[\underline{q}^r(.),\overline{q}^r(.)\right]&=\int_{a}^{b} \underline{L} \left(x,\underline{q}^r,\overline{q}^r,\dot{\underline{q}}^r,\dot{\overline{q}}^r\right)dx
\\&= \int_{\sigma_a}^{\sigma_b} \underline{L}^r \left(x(\sigma),\underline{q}^r(x((\sigma)),\overline{q}^r(x(\sigma)),
\frac{\frac{d\underline{q}^r(x(\sigma))}{d\sigma}}{\frac{dx(\sigma)}{d\sigma}},
\frac{\frac{d\overline{q}^r(x(\sigma))}{d\sigma}}{\frac{dx(\sigma)}{d\sigma}}\right)d\sigma\\
&= \int_{\sigma_a}^{\sigma_b} \underline{L}^r \left(x(\sigma),\underline{q}^r(x((\sigma)),\overline{q}^r(x(\sigma)),\frac{\dot{\underline{q}}^r_\sigma}{\dot{x}_\sigma},
\frac{\dot{\overline{q}}^r_\sigma}{\dot{x}_\sigma}\right) d\sigma\\
&\doteq \int_{\sigma_a}^{\sigma_b} \hat{\underline{L}^r }\left(x(\sigma),\underline{q}^r(x((\sigma)),\overline{q}^r(x(\sigma)),\dot{x}_\sigma,\dot{\underline{q}}^r_\sigma,
\dot{\overline{q}}^r_\sigma\right)d\sigma\\
&\doteq \hat{\underline{I}}^r \left[x(.),\underline{q}^r(x(.)),\overline{q}^r(x(.))\right],
\end{align*}
where $x(\sigma_a)=a,~x(\sigma_b)=b,~\dot{x}_\sigma=\frac{d x(\sigma)}{d\sigma},~\dot{
\underline{q}}^r_\sigma=\frac{d \underline{q}^r(x(\sigma))}{d\sigma}$ and $\dot{\overline{q}}^r_\sigma=\frac{d \overline{q}^r(x(\sigma))}{d\sigma}.$ By using similar arguments,
\begin{align*}
\overline{I}^r\left[\underline{q}^r(.),\overline{q}^r(.)\right]&
\\&\doteq \int_{\sigma_a}^{\sigma_b} \hat{\overline{L}^r }\left(x(\sigma),\underline{q}^r(x((\sigma)),\overline{q}^r(x(\sigma)),\dot{x}_\sigma,\dot{\underline{q}}^r_\sigma,\dot{\overline{q}}^r_\sigma\right)d\sigma
\\&\doteq \hat{\overline{I}}^r \left[x(.),\underline{q}^r(x(.)),\overline{q}^r(x(.))\right].
\end{align*}
If functional $\tilde{I}[\tilde{q}(x)]$ is invariant in the
sense of Definition \ref{q18}, then functional $ \hat{\tilde{I}}[\tilde{q}(x(\sigma))]$ is invariant in the sense of Definition \ref{q26}. Applying Theorem \ref{q27}, we obtain that
\begin{equation}\label{q30}
\frac{d}{dx}\underline{C}^r\left(x,\underline{q}^r,\overline{q}^r,\dot{x}_\sigma,\dot{\underline{q}}^r_\sigma,\dot{\overline{q}}^r_\sigma\right)=\frac{d}{dx}\left(\partial_4 \hat{\underline{L}}^r. \tau+\partial_5 \hat{\underline{L}}^r. \underline{\zeta}^r+\partial_6 \hat{\underline{L}}^r. \overline{\zeta}^r\right)=0
\end{equation}
and
\begin{equation}\label{q31}
\frac{d}{dx}\overline{C}^r\left(x,\underline{q}^r,\overline{q}^r,\dot{t}_\sigma,\dot{\underline{q}}^r_\sigma,\dot{\overline{q}}^r_\sigma\right)=\frac{d}{dx}\left(\partial_4 \hat{\overline{L}}^r. \tau+\partial_5 \hat{\overline{L}}^r. \underline{\zeta}^r+\partial_6 \hat{\overline{L}}^r. \overline{\zeta}^r\right)=0.
\end{equation}
 Since
\begin{align}\label{q28}
\partial_6 \hat{\underline{L}}^r&=\partial_5\underline{L}^r \left(x,\underline{q}^r,\overline{q}^r,\dot{\underline{q}}^r,\dot{\overline{q}}^r\right),\nonumber\\\partial_5 \hat{\underline{L}}^r&=\partial_4\underline{L}^r \left(x,\underline{q}^r,\overline{q}^r,\dot{\underline{q}}^r,\dot{\overline{q}}^r\right),\nonumber\\
\partial_4 \hat{\underline{L}}^r&=\underline{L}^r \left(x,\underline{q}^r,\overline{q}^r,\dot{\underline{q}}^r,\dot{\overline{q}}^r\right)-\partial_4\underline{L}^r \left(x,\underline{q}^r,\overline{q}^r,\dot{\underline{q}}^r,\dot{\overline{q}}^r\right).\left(\frac{\dot{\underline{q}}_\sigma^r}{\dot{x}_\sigma}\right)\nonumber\\&-\partial_5\underline{L}^r \left(x,\underline{q}^r,\overline{q}^r,\dot{\underline{q}}^r,\dot{\overline{q}}^r\right).\left(\frac{\dot{\overline{q}}_\sigma^r}{\dot{x}_\sigma}\right)\nonumber\\
&=\underline{L}^r \left(x,\underline{q}^r,\overline{q}^r,\dot{\underline{q}}^r,\dot{\overline{q}}^r\right)-\partial_4\underline{L}^r \left(x,\underline{q}^r,\overline{q}^r,\dot{\underline{q}}^r,\dot{\overline{q}}^r\right).  \dot{\underline{q}}^r-\partial_5\underline{L}^r \left(x,\underline{q}^r,\overline{q}^r,\dot{\underline{q}}^r.\dot{\overline{q}}^r\right).\dot{\overline{q}}^r\nonumber\\
\end{align}
and
\begin{align}\label{q29}
\partial_6 \hat{\overline{L}}^r&=\partial_5\overline{L}^r \left(x,\underline{q}^r,\overline{q}^r,\dot{\underline{q}}^r,\dot{\overline{q}}^r\right),\nonumber\\\partial_5 \hat{\overline{L}}^r&=\partial_4\overline{L}^r \left(x,\underline{q}^r,\overline{q}^r,\dot{\underline{q}}^r,\dot{\overline{q}}^r\right),\nonumber\\
\partial_4 \hat{\overline{L}}^r&=\overline{L}^r \left(x,\underline{q}^r,\overline{q}^r,\dot{\underline{q}}^r,\dot{\overline{q}}^r\right)-\partial_4\overline{L}^r \left(x,\underline{q}^r,\overline{q}^r,\dot{\underline{q}}^r,\dot{\overline{q}}^r\right).\left(\frac{\dot{\underline{q}}_\sigma^r}{\dot{x}_\sigma}\right)\nonumber\\&-\partial_5\overline{L}^r \left(x,\underline{q}^r,\overline{q}^r,\dot{\underline{q}}^r,\dot{\overline{q}}^r\right).\left(\frac{\dot{\overline{q}}_\sigma^r}{\dot{x}_\sigma}\right)\nonumber\\
&=\overline{L}^r \left(x,\underline{q}^r,\overline{q}^r,\dot{\underline{q}}^r,\dot{\overline{q}}^r\right)-\partial_4\overline{L}^r \left(x,\underline{q}^r,\overline{q}^r,\dot{\underline{q}}^r,\dot{\overline{q}}^r\right).  \dot{\underline{q}}^r-\partial_5\overline{L}^r \left(x,\underline{q}^r,\overline{q}^r,\dot{\underline{q}}^r.\dot{\overline{q}}^r\right).\dot{\overline{q}}^r\nonumber\\
\end{align}
substituting (\ref{q28}) and (\ref{q29}) into (\ref{q30}) and (\ref{q31}), we arrive to the intended conclusions (\ref{q32}) and (\ref{q33}).
\end{proof}

%-----------------------------------------------------------------------------------------

We now consider a more general case,  involving a delay in the Lagrangian function. To be more precise, consider the new problem:
\begin{equation}\label{zN1}
\tilde{I}_\tau[\tilde{q}(.)]=\int_{a}^{b} \tilde{L} (x,\tilde{q}(x),\dot{\tilde{q}}(x),\dot{\tilde{q}}(x-\tau))dx\longrightarrow \min,
\end{equation}
under the boundary conditions $\tilde{q}(x)=\tilde{\psi}(x)$, for all $x\in[a-\tau,a]$ and $\tilde{q}(b)=\tilde{q}_b$, with $0<\tau<b-a$ and
the Lagrangian $\underline{L}^r$ and $\overline{L}^r$  assumed to be $C^2$--functions with respect to all its arguments. We first prove necessary conditions in order to obtain a solution to the problem. To simplify the writing, we denote
\begin{equation}[x,\tilde q]^r:=\left(x,\underline{q}^r(x),\overline{q}^r(x),\dot{\underline{q}}^r(x),\dot{\overline{q}}^r(x),
\dot{\underline{q}}^r(x-\tau),\dot{\overline{q}}^r(x-\tau)\right).\end{equation}

\begin{theorem}\label{NT1}
If $\tilde{q}(x)$ is a solution for problem (\ref{zN1}), then it satisfies the fuzzy Euler--Lagrange equations:
\begin{equation}\label{NE1}
\begin{gathered}
\partial_2\underline{L}^r[x,\tilde q]^r-\frac{d}{dx}\left(\partial_4\underline{L}^r[x,\tilde q]^r+\partial_6\underline{L}^r[x+\tau,\tilde q]^r\right)=0,\\
\partial_3\underline{L}^r[x,\tilde q]^r-\frac{d}{dx}\left(\partial_5\underline{L}^r[x,\tilde q]^r+\partial_7\underline{L}^r[x+\tau,\tilde q]^r\right)=0,\\
\partial_2\overline{L}^r[x,\tilde q]^r-\frac{d}{dx}\left(\partial_4\overline{L}^r[x,\tilde q]^r+\partial_6\overline{L}^r[x+\tau,\tilde q]^r\right)=0,\\
\partial_3\overline{L}^r[x,\tilde q]^r-\frac{d}{dx}\left(\partial_5\overline{L}^r[x,\tilde q]^r+\partial_7\overline{L}^r[x+\tau,\tilde q]^r\right)=0,\\
\end{gathered}
\end{equation}
for all $x\in[a,b-\tau]$, and
\begin{equation}\label{NE2}
\begin{gathered}
\partial_2\underline{L}^r[x,\tilde q]^r-\frac{d}{dx}\partial_4\underline{L}^r[x,\tilde q]^r=0,\\
\partial_3\underline{L}^r[x,\tilde q]^r-\frac{d}{dx}\partial_5\underline{L}^r[x,\tilde q]^r=0,\\
\partial_2\overline{L}^r[x,\tilde q]^r-\frac{d}{dx}\partial_4\overline{L}^r[x,\tilde q]^r=0,\\
\partial_3\overline{L}^r[x,\tilde q]^r-\frac{d}{dx}\partial_5\overline{L}^r[x,\tilde q]^r=0,\\
\end{gathered}
\end{equation}
for all $x\in[b-\tau,b]$, and for all $r\in[0,1].$
\end{theorem}

\begin{proof}
Let $\epsilon \in \mathbb{R},$ and define a family of variations of the optimal solution of type  $\tilde{q}(x)+\epsilon \tilde{h}(x)$, where  $ \tilde{h}$ satisfies the boundary conditions $\tilde{h}(x)=\tilde{0}$, for all $x\in[a-\tau,a]$, and $\tilde{h}(b)=\tilde{0}$. Also, for convenience, we will assume that
$\tilde{h}(b-\tau)=\tilde{0}$. Let
\begin{equation}
\tilde{i}_\tau(\epsilon)=\int_{a}^{b} \tilde{L}\left(x,\tilde{q}(x)+\epsilon \tilde{q}(x),\dot{\tilde{q}}(x)+\epsilon \dot{\tilde{q}}(x),
\dot{\tilde{q}}(x-\tau)+\epsilon \dot{\tilde{q}}(x-\tau)\right)dx.
\end{equation}
The lower bound and upper bound of $\tilde{i}_\tau$ are respectively
\begin{align}
\underline{i}_\tau^r(\epsilon)=\int_{a}^{b}\underline{L}^r[x,\tilde q+\epsilon \tilde h]^r dx
\end{align}
and
\begin{align}
\overline{i}_\tau^r(\epsilon)=\int_{a}^{b}\overline{L}^r[x,\tilde q+\epsilon \tilde h]^r dx.
\end{align}
Differentiating $\underline{i}_\tau^r$ at $\epsilon=0$, we obtain
\begin{equation}\label{N1}\int_a^b\left(\partial_2\underline{L}^r[x,\tilde q]^r.\underline{h}^r(x)+\partial_3\underline{L}^r[x,\tilde q]^r.\overline{h}^r(x)
+\partial_4\underline{L}^r[x,\tilde q]^r.\dot{\underline{h}}^r(x)\right.\end{equation}
$$\left. +\partial_5\underline{L}^r[x,\tilde q]^r.\dot{\overline{h}}^r(x)
 +\partial_6\underline{L}^r[x,\tilde q]^r.\dot{\underline{h}}^r(x-\tau)+\partial_7\underline{L}^r[x,\tilde q]^r.\dot{\overline{h}}^r(x-\tau)
\right)dx=0.$$
Using integration by parts, we get that
\begin{equation}\int_a^b\partial_4\underline{L}^r[x,\tilde q]^r.\dot{\underline{h}}^r(x)dx=-\int_a^b\frac{d}{dx}\partial_4\underline{L}^r[x,\tilde q]^r.\underline{h}^r(x)dx\end{equation}
and
\begin{equation}\int_a^b\partial_5\underline{L}^r[x,\tilde q]^r.\dot{\overline{h}}^r(x)dx=-\int_a^b\frac{d}{dx}\partial_5\underline{L}^r[x,\tilde q]^r.\overline{h}^r(x)dx.\end{equation}
Also, integrating by parts and using the assumptions over $\tilde{h}$, we get
\begin{equation}\int_a^b\partial_6\underline{L}^r[x,\tilde q]^r.\dot{\underline{h}}^r(x-\tau)dx
=-\int_a^{b-\tau}\frac{d}{dx}\partial_6\underline{L}^r[x+\tau,\tilde q]^r.\underline{h}^r(x)dx\end{equation}
and
\begin{equation}\int_a^b\partial_7\underline{L}^r[x,\tilde q]^r.\dot{\overline{h}}^r(x-\tau)dx
=-\int_a^{b-\tau}\frac{d}{dx}\partial_7\underline{L}^r[x+\tau,\tilde q]^r.\overline{h}^r(x)dx.\end{equation}
Combining these relations into \eqref{N1}, and doing similar computations with respect to $\overline{i}_\tau^r$, we prove the result.
\end{proof}

The system of differential equations \eqref{NE1}-\eqref{NE2} are called Euler-–Lagrange equations with delay.

\begin{definition}[Invariance  without transforming time and with delay]
Functional $(\ref{zN1})$ is said to be invariant under an $\epsilon$--parameter group of infinitesimal transformations
\begin{equation}\label{qN1}
\tilde{\hat{q}}(x)=\tilde{q}(x)+\epsilon\tilde{\zeta}(x,\tilde{q}(x))+ o(\epsilon)
\end{equation}
if  and only if,
\begin{equation}\label{qN2}
\int_{t_a}^{t_b} \tilde{L} \left(x,\tilde{q}(x),\dot{\tilde{q}}(x),\dot{\tilde{q}}(x-\tau)\right)dx= \int_{t_a}^{t_b} \tilde{L} \left(x,\tilde{\hat{q}}(x),\dot{\tilde{\hat{q}}}(x),\dot{\tilde{\hat{q}}}(x-\tau)\right)dx,
\end{equation}
for any subinterval $[t_a,t_b]\subseteq [a,b]$ and $\epsilon>0.$
\end{definition}

Observe that since we impose the condition $\tilde{q}(x)=\tilde{\psi}(x)$, for all $x\in[a-\tau,a]$, then we have $\tilde{\zeta}(x,\tilde{q}(x))=0$ for all $x\in[a-\tau,a]$.

\begin{theorem}[Necessary condition of invariance with delay]\label{tN2} If functional \eqref{zN1} is invariant under transformations \eqref{qN1}, then for all $x\in[a,b-\tau]$,
\begin{equation*}
\partial_2 \underline{L}^r[x,\tilde q]^r.\underline{\zeta}^r(\cdot)+\partial_3 \underline{L}^r[x,\tilde q]^r.\overline{\zeta}^r(\cdot)+\partial_4 \underline{L}^r[x,\tilde q]^r.\dot{\underline{\zeta}}^r(\cdot)+\partial_5\underline{L}^r[x,\tilde q]^r.\dot{\overline{\zeta}}^r(\cdot)
\end{equation*}
\begin{equation}\label{qN13}
+\partial_6 \underline{L}^r[x+\tau,\tilde q]^r.\dot{\underline{\zeta}}^r(\cdot)+\partial_7\underline{L}^r[x+\tau,\tilde q]^r.\dot{\overline{\zeta}}^r(\cdot)=0,
\end{equation}
\begin{equation*}
\partial_2 \overline{L}^r[x,\tilde q]^r.\underline{\zeta}^r(\cdot)+\partial_3 \overline{L}^r[x,\tilde q]^r.\overline{\zeta}^r(\cdot)+\partial_4 \overline{L}^r[x,\tilde q]^r.\dot{\underline{\zeta}}^r(\cdot)+\partial_5\overline{L}^r[x,\tilde q]^r.\dot{\overline{\zeta}}^r(\cdot)
\end{equation*}
\begin{equation}\label{qN14}
+\partial_6 \overline{L}^r[x+\tau,\tilde q]^r.\dot{\underline{\zeta}}^r(\cdot)+\partial_7\overline{L}^r[x+\tau,\tilde q]^r.\dot{\overline{\zeta}}^r(\cdot)=0,
\end{equation}
and for all $x\in[b-\tau,b]$,
\begin{equation}\label{qNN13}
\partial_2 \underline{L}^r[x,\tilde q]^r.\underline{\zeta}^r(\cdot)+\partial_3 \underline{L}^r[x,\tilde q]^r.\overline{\zeta}^r(\cdot)+\partial_4 \underline{L}^r[x,\tilde q]^r.\dot{\underline{\zeta}}^r(\cdot)+\partial_5\underline{L}^r[x,\tilde q]^r.\dot{\overline{\zeta}}^r(\cdot)
\end{equation}
\begin{equation}\label{qNN14}
\partial_2 \overline{L}^r[x,\tilde q]^r.\underline{\zeta}^r(\cdot)+\partial_3 \overline{L}^r[x,\tilde q]^r.\overline{\zeta}^r(\cdot)+\partial_4 \overline{L}^r[x,\tilde q]^r.\dot{\underline{\zeta}}^r(\cdot)+\partial_5\overline{L}^r[x,\tilde q]^r.\dot{\overline{\zeta}}^r(\cdot)
\end{equation}
for all $r\in[0,1]$, where $(\cdot)=(x,\underline{q}^r(x),\overline{q}^r(x))$.
\end{theorem}

\begin{proof} Similar to the proof the Theorem \ref{t2}.
\end{proof}

\begin{definition}[Conserved quantity with delay]\label{defN0} Quantity $\tilde{C}(x,\tilde{q}(x),\dot{\tilde{q}}(x),\dot{\tilde{q}}(x-\tau))$ is said to be conserved if, and only if,
\begin{equation}
\frac{d}{dx}\underline{C}^r[x,\tilde q]^r=\frac{d}{dx}\overline{C}^r[x,\tilde q]^r=0,
\end{equation}
 along all the solutions of the Euler--Lagrange equations (\ref{NE1})-(\ref{NE2}) and  for all $r\in[0,1]$.
\end{definition}

\begin{theorem}
[Noether's theorem without transforming time and with delay] \label{qN27}
If functional \eqref{zN1} is invariant under
the one--parameter group of transformations \eqref{qN1}, then $\tilde{C}(x,\tilde{q}(x),\dot{\tilde{q}}(x),\dot{\tilde{q}}(x-\tau))$  is conserved where the lower and upper bound of  $\tilde{C}$ are
\begin{equation}\begin{array}{ll}
\displaystyle\underline{C}^r[x,\tilde q]^r
&\displaystyle=\partial_4 \underline{L}^r[x,\tilde q]^r.\underline{\zeta}^r(\cdot)+\partial_5 \underline{L}^r[x,\tilde q]^r
.\overline{\zeta}^r(\cdot)\\
&\quad\displaystyle+\partial_6 \underline{L}^r[x+\tau,\tilde q]^r.\underline{\zeta}^r(\cdot)+\partial_7 \underline{L}^r[x+\tau,\tilde q]^r
.\overline{\zeta}^r(\cdot)
\end{array}\end{equation}
and
\begin{equation}\begin{array}{ll}
\displaystyle\overline{C}^r[x,\tilde q]^r&
\displaystyle=\partial_4 \overline{L}^r[x,\tilde q]^r.\underline{\zeta}^r(\cdot)+\partial_5 \overline{L}^r[x,\tilde q]^r
.\overline{\zeta}^r(\cdot)\\
&\quad\displaystyle+\partial_6 \overline{L}^r[x+\tau,\tilde q]^r.\overline{\zeta}^r(\cdot)+\partial_7 \overline{L}^r[x+\tau,\tilde q]^r
.\overline{\zeta}^r(\cdot)
\end{array}\end{equation}
for $x\in[a,b-\tau]$, and for $x\in[b-\tau,b]$,
\begin{equation}
\underline{C}^r[x,\tilde q]^r=\partial_4 \underline{L}^r[x,\tilde q]^r.\underline{\zeta}^r(\cdot)+\partial_5 \underline{L}^r[x,\tilde q]^r
.\overline{\zeta}^r(\cdot)
\end{equation}
and
\begin{equation}
\overline{C}^r[x,\tilde q]^r=\partial_4 \overline{L}^r[x,\tilde q]^r.\underline{\zeta}^r(\cdot)+\partial_5 \overline{L}^r[x,\tilde q]^r
.\overline{\zeta}^r(\cdot)
\end{equation}
 for all $r\in[0,1]$, where $(\cdot)=(x,\underline{q}^r(x),\overline{q}^r(x))$.
\end{theorem}
\begin{proof} Similar to the proof of Theorem \ref{q27}.
\end{proof}

% --------------------------------------------

\section{example}\label{3}
\begin{example}
Let us consider the following fuzzy problem of the calculus of variations
\begin{equation*}
\tilde{I}[\tilde{q}(.)]=\int_{0}^1 x \dot{\tilde{q}}^2 dx \rightarrow \min.
\end{equation*}
We   derive the lower and upper bound  of $\tilde{I}$ under [(1) - gH]-
differentiability as follows:
\begin{equation*}
\underline{I}^r[\underline{q}^r(.),\overline{q}^r(.)]=\int_{0}^1 x \dot{\underline{q}}^2 dx =\int_{0}^1 \underline{L}^r dx,
\end{equation*}
\begin{equation*}
\overline{I}^r[\underline{q}^r(.),\overline{q}^r(.)]=\int_{0}^1 x \dot{\overline{q}}^2dt =\int_{0}^1 \overline{L}^r dx.
\end{equation*}
First, we show that   problem is invariant under    $(\tau, \tilde{\zeta})=(2x \ln x ,\tilde{q})$, for that we require
\begin{equation}\label{e1}
\int_{\hat{x}(a)}^{\hat{x}(b)} \hat{x} (\dot{\hat{\underline{q}}}^r)^2 d\hat{x}=\int_{a}^{b} x (\dot{\underline{q}}^r)^2 dx,
\end{equation}
where $$ \hat{x}=x+\epsilon 2x\ln x; ~~~~\hat{\underline{q}}^r=(1+\epsilon) \underline{q}^r.$$
Now since
$$\frac{d\hat{x}}{dx}=1+2\epsilon(\ln x+1),$$
$$\frac{d\hat{\underline{q}}^r}{d\hat{x}}=\frac{\frac{d\hat{\underline{q}}^r}{dx}}{\frac{d\hat{x}}{dx}}=\frac{\frac{d((1+\epsilon)\underline{q}^r)}{dx}}{\frac{d(x+\epsilon 2x\ln x)}{dx}}=\frac{(1+\epsilon)\dot{\underline{q}}^r}{1+2\epsilon+2\epsilon \ln x},$$
we have
\begin{equation*}
\begin{gathered}
\int_{\hat{x}(a)}^{\hat{x}(b)} \hat{x} (\dot{\hat{\underline{q}}}^r)^2 d\hat{x}=\int_{a}^{b} (x+\epsilon 2x\ln x)\left( \frac{(1+\epsilon)\dot{\underline{q}}^r}{1+2\epsilon+2\epsilon \ln x}\right)^2\frac{d\hat{x}}{dx} dx\\
=\int_{a}^{b} x\frac{(1+\epsilon 2\ln x)(1+\epsilon)^2}{1+2\epsilon+2\epsilon \ln x}(\dot{\underline{q}}^r)^2dx.
\end{gathered}
\end{equation*}
Since the term
\begin{gather*}
\frac{(1+\epsilon 2\ln x)(1+\epsilon)^2}{1+2\epsilon+2\epsilon \ln x}=\frac{(1+\epsilon 2\ln x)(1+2\epsilon+\epsilon^2)}{1+2\epsilon+2\epsilon \ln x}\\=\frac{1+2\epsilon+2\epsilon \ln x}{1+2\epsilon+2\epsilon \ln x}+O(\epsilon^2),
\end{gather*}
keeping only term of the first order of smallness in $\epsilon$ we have
\begin{equation}\label{e3}
\int_{\hat{x}(a)}^{\hat{x}(b)} \hat{x} (\dot{\hat{\underline{q}}}^r)^2 d\hat{x}=\int_{x(a)}^{x(b)} x(\dot{\underline{q}}^r)^2 dx.
\end{equation}
Once again, following the same arguments, one can easily show that
\begin{equation}\label{e4}
\int_{\hat{x}(a)}^{\hat{x}(b)} \hat{x} (\dot{\hat{\overline{q}}}^r)^2 d\hat{x}=\int_{x(a)}^{x(b)} x(\dot{\overline{q}}^r)^2 dx.
\end{equation}

Now, Noether's Theorem \ref{q36} indicates that
\begin{equation}\label{e5}
 x \underline{q}^r \dot{\underline{q}}^r - (\dot{\underline{q}}^r)^2 x^2 \ln x=const,
\end{equation}
\begin{equation}\label{e6}
 x \overline{q}^r \dot{\overline{q}}^r - (\dot{\overline{q}}^r)^2 x^2 \ln x=const,
\end{equation}
 along all the solutions of the Euler--Lagrange equations (\ref{q7}) and  for all $r\in[0,1]$. We can verify that the above expression is satisfied
for all extremals by differentiating the left-hand side of equations \eqref{e5}  and \eqref{e6} and
applying the Euler--Lagrange equations,
$$ \frac{d}{dx}(x\dot{\underline{q}}^r)=0,~~\frac{d}{dx}(x\dot{\overline{q}}^r)=0,$$
associated with the functional. In detail
\begin{equation}\label{e7}
\begin{gathered}
\frac{d}{dx}\left[(x\dot{\underline{q}}^r)\left(\underline{q}^r-x\dot{\underline{q}}^r\ln x\right)\right] \\=\left(\frac{d}{dx}(x\dot{\underline{q}}^r)\right)(\underline{q}^r-x\dot{\underline{q}}^r\ln x)
+(x\dot{\underline{q}}^r)\left(\dot{\underline{q}}^r-(\frac{d}{dx}(x\dot{\underline{q}}^r))\ln x-\dot{\underline{q}}^r\right)=0.
\end{gathered}
\end{equation}
Computing similar to those in equation \eqref{e7},  one can easily verify that
$$\frac{d}{dx}\left[(x\dot{\overline{q}}^r)(\overline{q}^r-x\dot{\overline{q}}^r\ln x)\right]=0.$$
\end{example}

\section{Conclusion}\label{4}
The standard approach to solve fuzzy problems of the calculus of variations is to make use of the necessary optimality
conditions given by the fuzzy Euler--Lagrange equations. These are, in general, nonlinear fuzzy differential equations,
very hard to be solved. One way to address the problem is to find conservation laws, i.e., quantities which are preserved
along the Euler--Lagrange extremals, and that can be used to simplify the problem at hands.\par
The main aim of our paper was to prove a Noether's theorem for fuzzy variational problems. As future work, we plan to study Noether theorem for fuzzy optimal control problems.

% ===================================================

\section*{Acknowledgments} R. Almeida was supported by Portuguese funds through the CIDMA - Center for Research and Development in Mathematics and Applications, and the Portuguese Foundation for Science and Technology (FCT-Funda\c{c}\~ao para a Ci\^encia e a Tecnologia), within project UID/MAT/04106/2013.

% ===================================================

\end{document}